\documentclass[11pt,a4paper,reqno]{amsart}
\usepackage{amsthm}
\usepackage{amssymb}
\usepackage{latexsym}
\usepackage{amsmath}
\usepackage{lscape}
\usepackage{euscript,url}
\usepackage{amssymb, latexsym, amsmath,enumerate,color}
\usepackage[english]{babel}
\usepackage[latin1]{inputenc}
\usepackage{times}
\usepackage[T1]{fontenc}
\usepackage[matrix,arrow,curve]{xy}
\usepackage{euscript}

\newtheorem{theorem}{Theorem}
\newtheorem{lemma}[theorem]{Lemma}
\newtheorem{prop}[theorem]{Proposition}

\theoremstyle{definition}

\theoremstyle{remark}

\newcommand{\cay}[2]{{\sf Cay}(#1,#2)}

\newcommand{\Z}{{\mathbb Z}}%{{\Zset}}

\newcommand{\aut}[1]{{\sf Aut}(#1)}

\newcommand{\F}{{\mathbb F}}

\newcommand{\comment}[1]{}

\newcommand{\rad}{\mathsf{Rad}}
\newcommand{\lcm}{\mathsf{l.c.m.}}

\title{A classification of  one dimensional affine rank three graphs}
\author{M. Muzychuk  }
\date{}

\begin{document}

\maketitle

\begin{abstract}
The rank three subgroups of a one-dimensional affine group over a finite field were classified in 1978 by Foulser and Kallaher. Although one can use their results for a classification of corresponding rank three graphs, the author didn't find such a classification neither in \cite{FK} nor in \cite{BM}. The goal of this note is to present such a classification. It turned out that graph classification is much simpler than the group one. More precisely, it is shown that the graphs in the title are either the Paley graphs or one of the graphs constructed by Van Lint and Schrijver or by  Peisert. Our approach 
is based on elementary group theory and does not use the classification of rank three affine groups.
\end{abstract}

%\section{Main result and its proof}

\vspace*{10mm}

 Recall that a {\it rank three graph} is an undirected graph $\Gamma$ the automorphism group of which has rank three. So, the arc set of such a graph coincides with one of 2-orbits (orbitals) of its automorphism group.

Let $\F_q$ be a finite field of order $q=p^r$ where $p$ is a prime and $r$ is a positive integer. A graph $\Gamma$ with point set $\F_q$ will be called a {\it one dimensional affine rank three graph} if $\aut{\Gamma}\cap A\Gamma L(1,q)$ is a rank three group. In other words, the arc set of $\Gamma$ coincides with one of 2-orbits of a rank three subgroup $G \leq A\Gamma L(1,q)$. 
 
A rank three subgroup $G \leq A\Gamma L(1,q)$
contains a normal abelian subgroup which coincides with the additive group $\F_q^+$ of the field. It is a semidirect product $G = \F_q^+ G_0$ where $G_0$ is the stablilizer of the zero $0\in\F_q$. 
The group $G_0$ is contained in $\Gamma L_1(q)\cong\F_q^*\rtimes\langle \phi\rangle$ where $\phi:x\mapsto x^p$ is the Frobenius automorphism. Note that $G$ has rank three if and only if $G_0$ has two orbits in its natural action on $\F_q^*$. Let us denote these orbits as $O_1$ and $O_2$. Then the non-reflexive 2-orbits (orbitals) of  $G$ are Cayley graphs over $\F_q^+$ the connection sets of which are $O_1$ and $O_2$. As was mentioned before, all subgroups of $\Gamma L_1(q)$ with two orbits on $\F_q^*$ were classified in \cite{FK}. 
But the orbits of these subgroups were not described there. 
Since orbit partitions of different groups may coincide, it could happen that the number of such partitions is much less than the number of groups. So, the classification of orbit partitions might be easier than the one of groups.
Surprisingly, this is indeed the case. In this note 
we provide a complete classification of the orbits of two orbit affine subgroups. Our approach is straightforward and is not based on the classification of affine rank three groups obtained in \cite{FK}.

%But different subgroups of $\Gamma L_1(q)$ may have the same orbits on $\F_q^*$ hereby providing the same rank three graphs. Let us call two subgroups of $\Gamma L_1(q)$ 1-equivalent if they have the same orbits on $\F_q^*$. Clearly that 1-equivalence is an equivalence relation, and, moreover, every equivalence class contains a unique maximal subgroup - the one generated by all subgroups from this class. 

In this note the group $\Gamma L_1(q)$ is considered as a permutation group
in its natural action on $\F_q^*$. It is generated by permutations $\phi$ and $\hat{\alpha},\alpha\in\F_q^*$ where $\hat{\alpha}$ stands for the mapping $x\mapsto x\alpha,x\in\F_q^*$. Note that the mapping $\alpha\mapsto\hat{\alpha}\in A\Gamma L_1(q)$ is a group monomorphism and its image $\widehat{\F_q^*}$ is a cyclic subgroup of $A\Gamma L_1(q)$ which acts regularly on $\F_q^*$.

We start with two main examples of rank three graphs. The first example was discovered in \cite{LS} and will be referred to as the Van Lint - Schrijver graph. It is described in the statement below which in fact coincides with Lemma 1 in \cite{LS}. We give here the proof to make the text self-contained.

\begin{prop}\label{200320b}
Let $p$ and $\ell$ be two primes satisfying ${\rm ord}_\ell(p)=\ell-1$.
Pick an arbitrary positive integer $k$ and set $q := p^{(\ell-1)k}$.
Let $\omega$ be a generator of $\F_q^*$. Then the subgroup $G_0=\langle \hat{\omega}^\ell, \phi\rangle\leq \Gamma L_1(q)$ has two orbits on $\F_q^*$:
$\langle \omega^\ell\rangle, \F_q^*\setminus \langle \omega^\ell\rangle$.
\end{prop}
\begin{proof}
It follows from $q= \left(p^{\ell-1}\right)^k$ that 
$p^{\ell-1} - 1$ divides $q-1$. Since $\ell$ divides $p^{\ell-1} - 1$, we conclude that $\ell\,|\,q-1$. Thus $C:=\langle \omega^\ell\rangle$ is a subgroup of $\F_q^*$ of index $\ell$. 

The subgroup $\widehat{C}$ has $\ell$ orbits on $\F_q^*$, namely:
%\footnote{Note that all these orbits are multiplicative cosets of $C$.} : 
$$C, C\omega,...,C\omega^{\ell-1}.$$
The permutation $\hat{\omega}^\ell=\widehat{\omega^\ell}$ fixes each of the orbits setwise while the action of $\phi$ on them is equivalent to the permutation of 
$i\mapsto pi,i\in\Z_\ell$. It follows from the assumption ${\rm ord}_\ell(p)=\ell-1$ that the permutation $i\mapsto pi,i\in\Z_\ell$ has two orbits on $\Z_\ell$, namely:$\{0\}$ and $\Z_\ell^*$. Therefore $\phi$ fixes the orbit $C=\langle\omega^\ell\rangle$ and permutes transitively the others. 
\end{proof}
{\bf Remark.} The subgroup $G_0=\langle \widehat{\omega^\ell}, \phi\rangle$ mentioned in Proposition~\ref{200320b} has a prime index $\ell$ in $\Gamma L_1(q)$, and, therefore, is maximal. The graph $\cay{\F_q^+}{C}$ is a particular case of  a generalized Paley graph \cite{KP}. Since $\cay{\F_q^+}{C}$ is a rank three graph, it is strongly regular. In the case of $\ell=2$ it coincides with the classical Paley graph, see e.g. \cite{J}.  If $\ell > 2$, then  the graph is a Latin square graph for odd values of $k$ and a Negative Latin Square graph if $k$ is even \cite{LS}.

\begin{prop}\label{200320c} Let $p\equiv\,3({\rm mod}\ 4)$ be a prime and $r$ an even positive integer. Denote by $\omega$ a generator of 
$\F_q^*$ where $q=p^r$. Then the subgroup $G_0=\langle \widehat{\omega^4},\phi\widehat{\omega}\rangle$ has two orbits on $\F_q^*$:
$C\cup C\omega$ and $C\omega^2\cup C\omega^3$ where $C:=\langle\omega^4\rangle$
\end{prop} 
\begin{proof} The subgroup $\widehat{C}=\langle\widehat{\omega^4}\rangle$ has $4$ orbits on $\F_q^*$: $C,C\omega,C\omega^2,C\omega^3$. The element $\phi\widehat{\omega}$ permutes these orbits in the following way: $C\leftrightarrow C\omega, C\omega^2\leftrightarrow C\omega^3$. Now the claim follows.
\end{proof}
\noindent{\bf Remark.} Note that the group $G$ in Proposition~\ref{200320c} has index $4$ in $\Gamma L_1(q)$. There are three ways to coarsen the partition  $C,C\omega,C\omega^2,C\omega^3$ into a partition with two classes of equal size:
$$
\{C\cup C\omega^2,C\omega\cup C\omega^3\},
\{C\cup C\omega,C\omega^2\cup C\omega^3\},
\{C\cup C\omega^3,C\omega^2\cup C\omega\}.
$$
The first partition is an orbit partition of the subgroup $\langle \widehat{\omega^2},\phi\rangle$ and yields the Paley graph, the second and the third one are orbit partitions of the groups $G_0$ and $G_0^\phi$, respectively. The latter two partitions were explicitly constructed by Peisert in \cite{P}. All three partitions yield strongly regular graphs with Paley parameters.

The theorem below proves that the examples described above are the only one dimensional affine rank three graphs.  
%\begin{theorem}\label{200320a} 
%Let $G_0\leq\Gamma L_1(q)$ be a subgroup which has two orbits $O_1$ and $O_2$ on the set of non-zero elements. Then $O_1,O_2$ are described in Propositions~\ref{200320b},\ref{200320c} and the remarks afterwards.
%\end{theorem}
%The proof of the Theorem will be done at the end of the note. This theorem implies   the following corollary which is the main result of the note.
%The corollary below is an  immediate consequence of Theorem~\ref{200320a}.
\begin{theorem}\label{280320a}
% Let $G\leq A\Gamma L_1(q)$ be a rank three subgroup and $\Gamma = (\F_q,E)$ a $G$-invariant strongly regular graph of valency at most $\frac{q-1}{2}$. Then either $\Gamma$ is a generalized Paley graph of valency $\frac{q-1}{\ell}$,% where $\ell$ is a prime divisor 
%$\ell$ is a prime,
Let $\Gamma$ be a one dimensional affine rank three graph with
point set $\F_q$. Then, up to a complement, $\Gamma$ is either the Van Lint - Schrijver, the Paley  or the Peisert graphs.
\end{theorem}

The group $\Gamma L_1(q)$ is a semidirect product $\F_q^*\rtimes\langle \phi\rangle$ which acts naturally on $\F_q^*$. So, we may consider a bit more general action, namely, the one of the semidirect product $X:=\Z_n\rtimes\langle\alpha\rangle,\alpha\in\aut{\Z_n}$  on $\Z_n$:
$$u^{(a,\alpha^i)} = \alpha^i (u) + a.$$ 
Pick an arbitrary 
subgroup $Y\leq X$ which has exactly two orbits on $\Z_n$, say $O_1$ and $O_2$ (we do not exclude the case when one of them is a singleton). 
Our goal is to describe all partitions  of $\Z_n$ appearing in this way. Then we apply the obtained classification to prove Theorem~\ref{280320a}.
%\ref{200320a}.
 
To formulate the key lemma we need one more definition. 
A {\it radical} $\rad(S)$ of a subset $S\subseteq \Z_n$ is defined as the subgroup of $\Z_n$ consisting of all $u\in\Z_n$ satisfying $S+u=S$, \cite{MP}.
Note that for any partition $\Z_n=O_1\cup O_2$ we have that $\rad(O_1)=\rad(O_2)$. 

Note that every subgroup $N$ of $\Z_n$ is normal in $X$. Its orbits coincide with its cosets and form an imprimitivity system for $X$. 
\begin{lemma}\label{140320b} Let $Y\leq X$ be a subgroup having two orbits on $\Z_n$, say $O_1$ and $O_2$,  $|O_1|\leq |O_2|$. Denote by $m$ the index of $\rad(O_1)$ in $\Z_n$.
Then one of the following holds
\begin{enumerate}
\item $m$ is prime, $O_1=m\Z_n$
 and 
$\alpha$ permutes transitively non-zero elements of the factor group $\Z_n/m\Z_n\cong\Z_m$.
\item $m=4$, $|O_1| = |O_2|$, $O_1$ is a union of two cosets of $4\Z_n$, $\alpha(x+4\Z_n) = -x+4\Z_n$  and $\{O_1,O_2\} = \{4\Z_n\cup i+4\Z_n,   2+4\Z_n\cup -i+4\Z_n\}, i\in\{1,3\}$.
\end{enumerate} 
\end{lemma}
\begin{proof} Write $\rad(O_1) = m\Z_n$.
If $m\Z_n$ is nontrivial, then we can replace the original action of $X$ on $\Z_n$ by the action on the cosets $\Z_n/m\Z_n\cong\Z_m$. The resulting permutation group will be $\overline{X}=\Z_{\overline{n}}\rtimes\langle\overline{\alpha}\rangle$ where $\overline{n}=m$ and $\overline{\alpha}$ is the automorphism of $\Z_m$ induced by $\alpha$ . The image $\overline{Y}$ has two orbits on $\Z_{\overline{n}}$, namely: $\overline{O_1}$ and $\overline{O_2}$ where $\overline{O_j}=O_j/m\Z_n$. Note that $\rad(\overline{O_j})$ is trivial now. This observation reduces the general case to the one with a trivial radical.

Thus starting from now we assume that $\rad(O_1)$ is trivial, i.e. $m=n$.
Pick an arbitrary prime divisor $p$ of $n$ and consider the subgroup $P:=\frac{n}{p}\Z_n,|P|=p$ . The cosets $i+P,i=0,...,n/p-1$ are blocks of $X$ and, therefore, are blocks of $Y$ too. Thus $(i+P)\cap O_j$ are blocks of the transitive action of $Y$ on $O_j,j=1,2$.
If some coset $i+P$ is contained in some $O_j$, then, it follows from transitivity of $Y$ on $O_j$ that $O_j$ is a union of $P$-cosets, contrary to the assumption $\rad(O_j)=\{0\}$. Hence, each coset $i+P$ intersects every orbit $O_j$ non-trivially. Since
$(i+P)\cap O_j$ is a block of $Y$ and $Y$ is transitive on $O_j$, the number $|(i+P)\cap O_j|$ depends only on $j$. Therefore $|O_j|=\frac{n}{p}|(i+P)\cap O_j|$, and, consequently, $\frac{n}{p}$ divides $|O_j|$. If $n$ has two distinct prime divisors, say $p$ and $q$, then 
$|O_j|$ is divisible by $\frac{n}{p}$ and $\frac{n}{q}$. Therefore $|O_j|$ is divisible by $\lcm(\frac{n}{p},\frac{n}{q})=n$, a contradiction. Thus $n=p^t$ for some prime $p$ and $|O_j| = p^{t-1} x_j$ with $x_1+x_2=p$ and $x_1\leq x_2$. 

Since $O_1$ has trivial radical, the group $Y\cap\Z_n$ is trivial. Therefore, $Y\cong Y\Z_n/\Z_n\hookrightarrow \langle \alpha\rangle$. In particular, $|Y|\leq |\langle\alpha\rangle|\leq\varphi(n) < n$. On the other hand $|Y|$ is divisible by $\lcm(|O_1|,|O_2|)=\frac{n}{p}\lcm(x_1,x_2)=\frac{n}{p}x_1x_2$. Hence $\frac{n}{p}x_1x_2 < n\implies x_1x_2 < p \implies x_1 =1, x_2=p-1$. Therefore 
$(p-1)\frac{n}{p} = \frac{n}{p}x_1x_2 \leq |Y| \leq\varphi(n)=\frac{n}{p}(p-1)\implies |Y|=\varphi(n)=\frac{n}{p}(p-1)\implies Y\cong\langle\alpha\rangle\cong\Z_n^*\implies X\cong\mathsf{Hol}(\Z_n)$.
 
Let $y$ be a generator of $Y$. Since $O_1$ and $O_2$ are the only orbits of $Y=\langle y\rangle$, the permutation $y$ has two cycles in its cyclic decomposition.
One of length $|O_1|$ and another one of length $|O_2|$. It follows from $|O_1|=p^{t-1},|O_2|=p^{t-1}(p-1)$ that $y^{p^{t-1}}$ acts trivially on $O_1$. 
Therefore $y^{p^{t-1}}$ has at least $p^{t-1}$ fixed points. 
If $p$ is odd, then 
$y^{p^{t-1}}$ is a non-trivial $p'$-element of $\mathsf{Hol}(\Z_n)$. Therefore it has only one fixed point, and, consequently, $t=1$.

Assume now that $p=2$. Then, as it was shown before, $\Z_{2^t}^*$ is cyclic. This is possible only when $t=1$ or $t=2$. The rest follows  easily.
\end{proof}

\noindent {\bf Proof of Theorem~\ref{280320a}.} As in Lemma~\ref{140320b} we
assume that $|O_1|\leq|O_2|$. Assume that case  (1) of the Lemma occurs. Then
$O_1$ is a subgroup of $\F_q^*$ of prime index $\ell$ and $\phi$ permutes transitively $\ell-1$ cosets $O_1\omega,...,O_1\omega^{\ell-1}$. Therefore the induced action of $\phi$ on the cosets:  
$O_1\omega^i\mapsto O_1\omega^{ip},i\in\Z_\ell^*$ is a full cycle. This implies that ${\rm ord}_\ell(p)=\ell-1$. Together with $p^r\equiv 1({\rm mod}\,\ell)$ this implies $\ell-1\,|\,r$. Thus $O_1$ is the set described in Proposition~\ref{200320b}.

Assume now that the second case of the Lemma~\ref{140320b} occurs. Then both $O_1$ and $O_2$ are unions of two cosets of an index $4$ subgroup $C\leq\F_q^*$ and $(C\omega)^\phi = C\omega^3$. The latter equality is equivalent to 
$C\omega^p = C\omega^3$, and, therefore, implies $p\equiv 3({\rm mod}\,4)$. Together with $p^r\equiv 1({\rm mod}\,4)$ this yields $2\,|\,r$. This yields us the  examples described in Proposition~\ref{200320c}.\hfill $\square$

\vspace*{10mm}

\centerline{\bf Acknowledgments}

\vspace*{5mm}

The author is very grateful to I. Ponomarenko who attracted the author's attention to  the topic and read carefully the first draft. I am also thankful to A. Vasil'ev for careful reading and to M. Klin for useful discussions.

\end{document}